\documentclass{amsart}

\usepackage{amsmath,amsthm,amssymb}

\theoremstyle{plain}

\newtheorem{definition}{Definition}
\newtheorem{thm}[definition]{Theorem}

\newtheorem{lem}[definition]{Lemma}
\newtheorem{cor}[definition]{Corollary}
\newtheorem{rei}[definition]{Example}

\begin{document}
\title[An extension and a generalization of Dedekind's theorem]{An extension and a generalization of Dedekind's theorem}
\author[N. Yamaguchi]{Naoya Yamaguchi}
\date{\today}

\maketitle

\begin{abstract}
For any given finite abelian group, we give factorizations of the group determinant in the group algebra of any subgroup. 
The factorizations are an extension of Dedekind's theorem. 
The extension leads to a generalization of Dedekind's theorem and a simple expression for inverse elements in the group algebra. 
\end{abstract}

\section{\bf{INTRODUCTION}}
In this paper, we give factorizations of the group determinant for any given finite abelian group $G$ in the group algebra of subgroups. 
The factorizations are an extension of Dedekind's theorem. 
The extension leads to a generalization of Dedekind's theorem and a simple expression for inverse elements in the group algebra.

The group determinant $\Theta(G)$ is the determinant of a matrix whose elements are independent variables $x_{g}$ corresponding to $g \in G$. 
Dedekind gave the following theorem about the irreducible factorization of the group determinant for any finite abelian group. 

\begin{thm}[Dedekind \cite{conrad1998origin}]\label{thm:1.1.1}
Let $G$ be a finite abelian group and $\widehat{G}$ the group of characters of $G$. 
Then we have 
$$
\Theta(G) = \prod_{\chi \in \widehat{G}} \sum_{g \in G} \chi (g) x_{g}. 
$$
\end{thm}

Frobenius gave the following theorem about the irreducible factorization of the group determinant for any finite group; 
thus, Frobenius gave a generalization of Dedekind's theorem.

\begin{thm}[Frobenius \cite{conrad1998origin}]\label{thm:1.1.2}
Let $G$ be a finite group and $\widehat{G}$ a complete set of irreducible representations of $G$ over $\mathbb{C}$. 
Then we have 
$$
\Theta(G) = \prod_{\varphi \in \widehat{G}} \det{\left( \sum_{g \in G} \varphi (g) x_{g} \right)^{\deg{\varphi}}}. 
$$
\end{thm}

The main results of this paper are an extension and a generalization of Dedekind's theorem that are different from Frobenius' theorem.

\subsection{Main results}
We give an extension and a generalization of Dedekind's theorem.

Let $G$ be a finite abelian group, $\mathbb{C} G$ the group algebra of $G$ over $\mathbb{C}$, 
$\mathbb{C}[x_{g}] = \mathbb{C}\left[x_{g}; g \in G \right]$ the polynomial ring in $\left\{ x_{g} \: \vert \: g \in G \right\}$ with coefficients in $\mathbb{C}$, 
$
\mathbb{C}[x_{g}] G 
= \mathbb{C}[x_{g}] \otimes \mathbb{C}G 
= \left\{ \sum_{g \in G} A_{g} g \: \vert \: A_{g} \in \mathbb{C}[x_{g}] \right\}
$ 
the group algebra of $G$ over $\mathbb{C}[x_{g}]$, 
$H$ a subgroup of $G$, 
and $[G : H]$ the index of $H$ in $G$. 
Then we have the following theorem that is an extension of Dedekind's theorem. 

\begin{thm}\label{thm:1.1.3}
Let $G$ be a finite abelian group, $e$ the unit element of $G$, 
$H$ a subgroup of $G$, 
and $\widehat{H}$ the dual group of $H$. 
For every $h \in H$, there exists a homogeneous polynomial $A_{h} \in \mathbb{C}[x_{g}]$ such that $\deg{A_{h}} = [G : H]$ and 
$$
\Theta(G) e = \prod_{\chi \in \widehat{H}} \sum_{h \in H} \chi(h) A_{h} h. 
$$
If $H = G$, we can take $A_{h} = x_{h}$ for each $h \in H$. 
\end{thm} 

Note that the equality in Theorem~$\ref{thm:1.1.3}$ is the equality in $\mathbb{C}[x_{g}] H$. 
Theorem~$\ref{thm:1.1.3}$ leads to the following theorem.

\begin{thm}\label{thm:1.1.4}
Let $G$ be a finite abelian group and $H$ a subgroup of $G$. 
For every $h \in H$, there exists a homogeneous polynomial $A_{h} \in \mathbb{C}[x_{g}]$ such that $\deg{A_{h}} = \left[G : H \right]$ and 
$$
\Theta(G) = \prod_{\chi \in \widehat{H}} \sum_{h \in H} \chi(h) A_{h}. 
$$
If $H = G$, we can take $A_{h} = x_{h}$ for each $h \in H$. 
\end{thm}

Theorem~$\ref{thm:1.1.4}$ is a generalization of Dedekind's theorem. 
In fact, let $H = G$ and $A_{h} = x_{h}$. Then we have Dedekind's theorem.

Moreover, we obtain the following formula for inverse elements in the group algebra $\mathbb{C}G$ from Theorem~$\ref{thm:1.1.3}$.  
However, only now the situation is that $x_{g}$ is a complex number for any $g \in G$. 
Hence, we assume that $\sum_{g \in G} x_{g} g \in \mathbb{C}G$ and $\Theta(G) = \det{(x_{gh^{-1}})_{g, h \in G}} \in \mathbb{C}$. 

\begin{cor}\label{cor:1.1.5}
Let $G$ be a finite abelian group, $\chi_{1}$ the trivial representation of $G$, and $\sum_{g \in G} x_{g}g \in \mathbb{C}G$ such that $\Theta(G) \neq 0$. 
Accordingly, we have 
$$
\left( \sum_{g \in G} x_{g} g \right)^{-1} = \frac{1}{\Theta(G)} \prod_{\chi \in \widehat{G} \setminus \{ \chi_{1} \} } \left( \sum_{g \in G} \chi (g) x_{g} g \right). 
$$
\end{cor}

\section{\bf{Irreducible factorization of group determinant}}
In this section, we recall the definition of the group determinant and its irreducible factorization.

\subsection{Irreducible factorization of group determinant}
Let $G$ be a finite group and $\left\{ x_{g} \: \vert \: g \in G \right\}$ independent commuting variables. 
Below, we define the group determinant $\Theta(G)$ of $G$.

\begin{definition}\label{def:2.1.1}
The group determinant $\Theta(G)$ of $G$ is given by 
$$
\Theta(G) = \det{\left( x_{g h^{-1}} \right)}_{g, h \in G}
$$
where we give a numbering to the element of $G$. 
\end{definition}

Namely, the group determinant $\Theta(G)$ is a homogeneous polynomial of degree $|G|$ in $\{ x_{g} \: \vert \: g \in G \}$, 
where $|G|$ is the order of $G$.

In general, the matrix ${\left( x_{g h^{-1}} \right)}_{g, h \in G}$ is a covariant under change of a numbering to the element of $G$. 
However, the group determinant $\Theta(G)$ is an invariant.

\begin{rei}[]\label{rei:2.1.2}
Let $G = \mathbb{Z} / 3 \mathbb{Z} = \{ 0, 1, 2 \}$. 
Then we have 
\begin{align*}
\Theta(G) &= \det{ 
\begin{bmatrix} 
x_{0} & x_{2} & x_{1} \\ 
x_{1} & x_{0} & x_{2} \\ 
x_{2} & x_{1} & x_{0}
\end{bmatrix}
}. 
\end{align*}
\end{rei}

Dedekind proved the following theorem about the irreducible factorization of the group determinant for any finite group. 
\begin{thm}[Dedekind \cite{conrad1998origin}]\label{thm:2.1.3}
Let $G$ be a finite group and $\widehat{G}$ the group of characters of $G$. 
Then we have 
$$
\Theta(G) = \prod_{\chi \in \widehat{G}} \sum_{g \in G} \chi (g) x_{g}. 
$$
\end{thm}

\begin{rei}[]\label{rei:2.1.4}
Let $G = \mathbb{Z} / 3 \mathbb{Z} = \{ 0, 1, 2 \}$. 
Then we have 
\begin{align*}
\Theta(G) &= \det{ 
\begin{bmatrix} 
x_{0} & x_{2} & x_{1} \\ 
x_{1} & x_{0} & x_{2} \\ 
x_{2} & x_{1} & x_{0}
\end{bmatrix} 
} \\ 
&= (x_{0} + x_{1} + x_{2}) (x_{0} + x_{1} \omega + x_{2} \omega^{2} ) (x_{0} + x_{1} \omega^{2} + x_{2} \omega) 
\end{align*}
where $\omega$ is a primitive third root of unity. 
\end{rei}

\section{\bf{An extension and a generalization of Dedekind's theorem}}
In this section, we give an extension and a generalization of Dedekind's theorem.

\subsection{Degree one representations}
In this subsection, we describe two lemmas needed later.

Let $G$ be a finite group, $\overline{G}$ the set of degree one representations, $H$ a subgroup of $G$ and
$$
\overline{G}_{H} = \left\{ \chi \in \overline{G} \: \vert \: \chi(h) = 1, h \in H \right\}. 
$$
Then, $\overline{G}_{H}$ is a subgroup of $\overline{G}$.

Let $\widehat{G}$ be a complete set of irreducible representations of $G$. 
If $G$ is an abelian group, since the degree of irreducible representations of $G$ is one, we have $\overline{G} = \widehat{G}$.

The following lemmas are well known.

\begin{lem}\label{lem:3.1.1}
Let $G$ be a finite group and $H$ a normal subgroup of $H$ such that $G/H$ is an abelian group. 
Then we have 
$$
\overline{G}_{H} = \left\{ \varphi \circ \pi \: \vert \: \varphi \in \widehat{G/H} \right\}
$$
where $\pi : G \rightarrow G/H$ is a natural projection. 
\end{lem} 
\begin{proof}
Clearly, $\left\{ \varphi \circ \pi \: \vert \: \varphi \in \widehat{G/H} \right\} \subset \overline{G}_{H}$. 
We show that $\overline{G}_{H} \subset \left\{ \varphi \circ \pi \: \vert \: \varphi \in \widehat{G/H} \right\}$. 
Let $\chi \in \overline{G}_{H}$. We define the map $\varphi : G/H \rightarrow \mathbb{C}^{\times}$ by $\varphi(gH) = \chi(g)$. 
It is easy to see that $\varphi$ is well defined and $\chi = \varphi \circ \pi$. 
This completes the proof. 
\end{proof}

\begin{lem}\label{lem:3.1.2}
Let $G$ be a finite abelian group, and suppose that $g \in G$ is not the unit element of $G$. 
Then, there exists $\chi \in \widehat{G}$ such that $\chi(g) \neq 1$. 
\end{lem} 
\begin{proof}
From the structure theorem for finite abelian groups, 
there exist cyclic groups $\mathbb{Z}/ m_{i} \mathbb{Z} \: (1 \leq i \leq r)$ and a group isomorphism 
$$
f : G \rightarrow \mathbb{Z}/m_{1} \mathbb{Z} \times \mathbb{Z}/m_{2} \mathbb{Z} \times \cdots \times \mathbb{Z}/m_{r} \mathbb{Z}. 
$$
Therefore, for all $g \in G$, there exists $\overline{a_{i}} \in \mathbb{Z}/ m_{i} \mathbb{Z}$ such that 
$$
f(g) = (\overline{a_{1}}, \overline{a_{2}}, \ldots, \overline{a_{r}}). 
$$
For all $x_{i} \in \mathbb{N} \: (1 \leq i \leq r)$ where we assume that $0 \in \mathbb{N}$, 
we define the map $\chi : G \rightarrow \mathbb{C}^{\times}$ by 
$$
\chi(g) = \xi_{1}^{x_{1} a_{1}} \xi_{2}^{x_{2} a_{2}} \cdots \xi_{r}^{x_{r} a_{r}} 
$$
where $\xi_{i}$ is a primitive $m_{i}$-th root of unity($1 \leq i \leq r$). 
Then, the map $\chi$ is a degree one representation of $G$. 
Since $g$ is not the unit element, there exists $i \neq 0$ such that $a_{i} \neq 0$. 
Let $x_{i} = 1$ and $x_{j} = 0 \: (1 \leq i \neq j \leq r)$. 
Then, $\chi$ is a degree one representation of $G$ such that $\chi(g) \neq 1$. 
This completes the proof. 
\end{proof}

\begin{lem}\label{lem:3.1.3}
Let $G$ be a finite group and $H$ a normal subgroup of $G$ such that $G/H$ is an abelian group. 
If $g \not\in H$, there exists $\chi \in \overline{G}_{H}$ such that $\chi(g) \neq 1$. 
\end{lem} 
\begin{proof}
From Lemma~$\ref{lem:3.1.2}$, there exists $\varphi \in \widehat{G/H}$ such that $\varphi(gH) \neq 1$ where $g \not\in H$. 
Let $\pi : G \rightarrow G/H$ be the natural projection. 
By Lemma~$\ref{lem:3.1.1},\ \chi = \varphi \circ \pi \in \overline{G}_{H}$. 
This completes the proof. 
\end{proof}

\subsection{Operators on group algebras}
In this subsection, we define operators on group algebras that are used in the proof of the main theorem.

\begin{definition}\label{def:3.2.1}
Let $G$ be a finite group and $\chi \in \overline{G}$. 
We define the map $T_{\chi} : \mathbb{C}[x_{g}] G \rightarrow \mathbb{C}[x_{g}] G$ by 
$$
T_{\chi} \left( \sum_{g \in G} A_{g} g \right) = \sum_{g \in G} \chi(g) A_{g} g
$$
where $A_{g} \in \mathbb{C}[x_{g}]$. 
\end{definition}

Let $\chi, \chi' \in \overline{G}$ and $\alpha, \beta \in \mathbb{C}[x_{g}] G$. 
It is easy to see that $T_{\chi} \circ T_{\chi'} = T_{\chi \circ \chi'}$ and $T_{\chi}(\alpha \beta) = T_{\chi}(\alpha) T_{\chi}(\beta)$, 
where $\left( \chi \circ \chi' \right)(g) = \chi(g) \chi'(g)$.

We give a necessary and sufficient condition for $T_{\chi}$-invariance for all $\chi \in \overline{G}_{H}$.

\begin{lem}\label{lem:3.2.2}
Let $G$ be a finite group, $H$ a normal subgroup of $G$ such that $G/H$ is an abelian group and $\alpha \in \mathbb{C}[x_{g}] G$. 
For all $\chi \in \overline{G}_{H},\ T_{\chi}(\alpha) = \alpha$ if and only if $\alpha \in \mathbb{C}[x_{g}] H$. 
\end{lem} 
\begin{proof}
Let $\alpha \in \mathbb{C}[x_{g}] H$. 
Obviously, $T_{\chi}(\alpha) = \alpha$ for all $\chi \in \overline{G}_{H}$. 
Let $\alpha = \sum_{g \in G} A_{g} g$. 
If $T_{\chi}(\alpha) = \alpha$ for all $\chi \in \overline{G}_{H}$, 
then we have $\chi(g)A_{g}g = A_{g} g$ for all $g \in G$. 
From this condition and Lemma~$\ref{lem:3.1.3}$, if $g \not\in H$, there exists $\chi \in \overline{G}_{H}$ such that $\chi(g) \neq 1$. 
Therefore, $A_{g} = 0$. 
Namely, $\alpha = \sum_{h \in H} A_{h} h$. 
This completes the proof. 
\end{proof}

Let $G$ be a finite group, $S$ a subgroup of $\widehat{G}$ and $S|_{H}$ the set of restrictions of $\chi \in S$ on $H$.

\begin{lem}\label{lem:3.2.3}
Let $G$ be a finite abelian group, 
$H$ a subgroup of $G$, 
and $\widehat{G} = \chi_{1} \widehat{G}_{H} \sqcup \chi_{2} \widehat{G}_{H} \sqcup \cdots \sqcup \chi_{k} \widehat{G}_{H}$. 
Then we have $k = |H|$ and $\widehat{H} = \{ \chi_{1}, \chi_{2}, \ldots, \chi_{k} \}|_{H}$. 
\end{lem} 
\begin{proof}
First, we show that $k = |H|$. 
From $|G| = | \widehat{G} | = k | \widehat{G}_{H} |$ and Lemma~\ref{lem:3.1.1}, we have $|\widehat{G}_{H}| = |\widehat{G/H}| = \frac{|G|}{|H|}$. 
Therefore, $k = |H|$. 
Next, we show that $\widehat{H} = \{ \chi_{1}, \chi_{2}, \ldots, \chi_{k} \}|_{H}$. 
Since the restriction of elements of $\widehat{G}_{H}$ is the trivial representation on $H$, 
$\widehat{G}|_{H} = \{ \chi_{1}, \chi_{2}, \ldots, \chi_{k} \}|_{H} \subset \widehat{H}$. 
From $| \widehat{H} | = \left| H \right|$, we can show that $\chi_{1}, \chi_{2}, \ldots, \chi_{k}$ are different on $H$. 
If $\chi_{i}(h) = \chi_{j}(h) \: (1 \leq i \neq j \leq k)$ for all $h \in H,\ \left( \chi_{i}^{-1} \circ \chi_{j} \right) (h) = 1$. 
Therefore, $\chi_{i}^{-1} \circ \chi_{j} \in \widehat{G}_{H}$. 
This is a contradiction for the left $\widehat{G}_{H}$-coset decomposition of $\widehat{G}$. 
Namely, we have $\chi_{i} \neq \chi_{j}$. 
This completes the proof. 
\end{proof}

\subsection{An extension and a generalization of Dedekind's theorem}

In this subsection, we give the extension and generalization of Dedekind's theorem.

\begin{lem}\label{lem:3.3.1}
Let $G$ be a finite abelian group, $e$ the unit element of $G$, and $H$ a subgroup of $G$. 
For every $h \in H$, there exists a homogeneous polynomial $A_{h} \in \mathbb{C}[x_{g}]$ such that $\deg{A_{h}} = [G : H]$ and 
$$
\prod_{\chi \in \widehat{G}_{H}} \sum_{g \in G} \chi(g) x_{g} g = \sum_{h \in H} A_{h} h
$$
If $H = G$, we can take $A_{h} = x_{h}$ for each $h \in H$. 
\end{lem} 
\begin{proof}
For all $\chi' \in \widehat{G}_{H}$, 
\begin{align*}
T_{\chi'} \left( \prod_{\chi \in \widehat{G}_{H}} \sum_{g \in G} \chi(g) x_{g} g \right) 
&= \prod_{\chi \in \widehat{G}_{H}} \sum_{g \in G} \left( \chi' \circ \chi \right) (g) x_{g} g \\ 
&= \prod_{\chi \in \widehat{G}_{H}} \sum_{g \in G} \chi(g) x_{g} g. 
\end{align*}
From Lemma~\ref{lem:3.2.2}, we have $\prod_{\chi \in \widehat{G}_{H}} \sum_{g \in G} \chi(g) x_{g} g \in \mathbb{C}[x_{g}] H$. 
Clearly, $\deg{A_{h}} = |\widehat{G}_{H}| = \left[G : H \right]$. 
If $H = G,\ \widehat{G}_{H}$ is the trivial group. 
This completes the proof. 
\end{proof}

\begin{definition}\label{def:3.3.2}
Let $F : \mathbb{C}[x_{g}] G \rightarrow \mathbb{C}[x_{g}]$ be the $\mathbb{C}[x_{g}]$-algebra homomorphism such that $F(g) = 1$ for all $g \in G$. 
We call the map F the fundamental $\mathbb{C}[x_{g}]G$-function. 
\end{definition}

Now, we give factorizations of the group determinant for any given finite abelian group in the group algebra of subgroups. 
The factorizations are the extension of Dedekind's theorem.

\begin{thm}\label{thm:3.3.3}
Let $G$ be a finite abelian group, $e$ the unit element of $G$, and $H$ a subgroup of $G$. 
For every $h \in H$, there exists a homogeneous polynomial $A_{h} \in \mathbb{C}[x_{g}]$ such that $\deg{A_{h}} = \left[G : H \right]$ and 
$$
\Theta(G) e = \prod_{\chi \in \widehat{H}} \sum_{h \in H} \chi(h) A_{h} h. 
$$
If $H = G$, we can take $A_{h} = x_{h}$ for each $h \in H$. 
\end{thm} 
\begin{proof}
Clearly, 
$$
T_{\chi} \left( \prod_{\chi \in \widehat{G}} \sum_{g \in G} \chi(g) x_{g} g \right) = \prod_{\chi \in \widehat{G}} \sum_{g \in G} \chi(g) x_{g} g 
$$
for all $\chi \in \widehat{G}$. From this, $\widehat{G} = \widehat{G}_{\{ e \}}$ and Lemma~$\ref{lem:3.2.2}$, there exists $C \in \mathbb{C}[x_{g}]$ such that 
\begin{align*}
\prod_{\chi \in \widehat{G}} \sum_{g \in G} \chi(g) x_{g} g 
&= \prod_{\chi \in \widehat{G}_{\{ e \}}} \sum_{g \in G} \chi(g) x_{g} g \\ 
&= C e. 
\end{align*}
Let $F$ be the fundamental $\mathbb{C}[x_{g}]G$-function. 
By applying $F$ to this equation and Theorem~$\ref{thm:2.1.3}$, 
we have $C = \Theta(G)$. 
Namely, we have 
$$
\prod_{\chi \in \widehat{G}} \sum_{g \in G} \chi(g) x_{g} g = \Theta(G) e. 
$$
Let $\widehat{G} = \chi_{1} \widehat{G}_{H} \sqcup \chi_{2} \widehat{G}_{H} \sqcup \cdots \sqcup \chi_{k} \widehat{G}_{H}$. 
Then we have 
\begin{align*}
\prod_{\chi \in \widehat{G}} \sum_{g \in G} \chi(g) x_{g} g 
&= \prod_{i = 1}^{k} \prod_{\chi \in \chi_{i} \widehat{G}_{H}} \sum_{g \in G} \chi(g) x_{g} g \\ 
&= \prod_{i = 1}^{k} T_{\chi_{i}} \left( \prod_{\chi \in \widehat{G}_{H}} \sum_{g \in G} \chi(g) x_{g} g \right). 
\end{align*}
There exists a homogeneous polynomial $A_{h} \in \mathbb{C}[x_{g}]$ for each $h \in H$ such that 
\begin{align*}
\prod_{i = 1}^{k} T_{\chi_{i}} \left( \prod_{\chi \in \widehat{G}_{H}} \sum_{g \in G} \chi(g) x_{g} g \right) 
&= \prod_{i = 1}^{k} T_{\chi_{i} \vert_{H}} \left( \sum_{h \in H} A_{h} h \right) \\ 
&= \prod_{\chi \in \widehat{H}} \sum_{h \in H} \chi(h) A_{h} h 
\end{align*}
from Lemma~$\ref{lem:3.2.3}$ and $\ref{lem:3.3.1}$. 
This completes the proof. 
\end{proof}

As a corollary, we obtain the following formula for inverse elements in the group algebra $\mathbb{C}G$ from Theorem~$\ref{thm:1.1.3}$.  
However, only now the situation is that $x_{g}$ is a complex number for any $g \in G$. 
Hence, we assume that $\sum_{g \in G} x_{g} g \in \mathbb{C}G$ and $\Theta(G) = \det{(x_{gh^{-1}})_{g, h \in G}} \in \mathbb{C}$. 

\begin{cor}\label{cor:1.1.5}
Let $G$ be a finite abelian group, $\chi_{1}$ the trivial representation of $G$, and $\sum_{g \in G} x_{g}g \in \mathbb{C}G$ such that $\Theta(G) \neq 0$. 
Then we have 
$$
\left( \sum_{g \in G} x_{g} g \right)^{-1} = \frac{1}{\Theta(G)} \prod_{\chi \in \widehat{G} \setminus \{ \chi_{1} \} } \left( \sum_{g \in G} \chi (g) x_{g} g \right). 
$$
\end{cor}

We give factorizations of the group determinant for any given finite abelian group. 
The factorizations are the generalization of Dedekind's theorem.

\begin{thm}\label{thm:3.3.5}
Let $G$ be a finite abelian group and $H$ a subgroup of $G$. 
For every $h \in H$, there exists a homogeneous polynomial $A_{h} \in \mathbb{C}[x_{g}]$ such that $\deg{A_{h}} = \left[G : H \right]$ and 
$$
\Theta(G) = \prod_{\chi \in \widehat{H}} \sum_{h \in H} \chi(h) A_{h}. 
$$
If $H = G$, we can take $A_{h} = x_{h}$ for each $h \in H$. 
\end{thm} 
\begin{proof}
From Theorem~$\ref{thm:3.3.3}$ and the fundamental $\mathbb{C}[x_{g}]G$-function, we have 
$$
\Theta(G) = \prod_{\chi \in \widehat{H}} \sum_{h \in H} \chi(h) A_{h}. 
$$
This completes the proof. 
\end{proof}

\clearpage

\thanks{\bf{Acknowledgments}}
I am deeply grateful to Prof. Hiroyuki Ochiai and Prof. Minoru Itoh who provided the helpful comments and suggestions. 
Also, I would like to thank my colleagues in the Graduate School of Mathematics of Kyushu University, 
in particular Cid Reyes, Tomoyuki Tamura and Yuka Suzuki for comments and suggestions. 
I would also like to express my gratitude to my family for their moral support and warm encouragements. 
This work was supported by a grant from the Japan Society for the Promotion of Science (JSPS KAKENHI Grant Number 15J06842).

\medskip
\begin{flushleft}
Naoya Yamaguchi\\
Graduate School of Mathematics\\
Kyushu University\\
Nishi-ku, Fukuoka 819-0395 \\
Japan\\
n-yamaguchi@math.kyushu-u.ac.jp
\end{flushleft}

\end{document}